\documentclass[12 pt]{amsart}

\usepackage{graphicx}
\usepackage{stmaryrd}
\usepackage{amssymb}
\usepackage[all]{xy}
\CompileMatrices

\newcommand{\A}{\mathcal{A}}
\newcommand{\B}{\mathcal{B}}

\renewcommand{\P}{\mathbb{P}}

\newcommand{\rng}{\mbox{rng}}

\renewcommand{\rng}{\mbox{rng}}

\newcommand{\ana}{\mathbf{\Sigma}^1_1}
\newcommand{\coana}{\mathbf{\Pi}^1_1}
\newcommand{\lana}{\Sigma^1_1}
\newcommand{\lcoana}{\Pi^1_1}

\newcommand{\lclosed}{\Pi^0_1}
\newcommand{\baire}{\omega^\omega}
\newcommand{\cantor}{2^\omega}
\newcommand{\gdelta}{\mathbf{G}_\delta}
\newcommand{\fsigma}{\mathbf{F}_\sigma}

\newcommand{\sh}{\mathbf{\Sigma}^1_2}

\newcommand{\dualsh}{\mathbf{\Delta}^1_2}
\newcommand{\ldsh}{\Delta^1_2}
\newcommand{\lsh}{\Sigma^1_2}

\newcommand{\ctree}{2^{<\omega}}
\newcommand{\btree}{\omega^{<\omega}}
\newcommand{\ellentuck}{[\omega]^\omega}

\newcommand{\dom}{\mathrm{dom}}

\newcommand{\continuous}{\xrightarrow{\mbox{ }c\mbox{ }}}
\newcommand{\cupdot}{\sqcup}
\newcommand{\bigcupdot}{\bigsqcup}
\renewcommand{\P}{\mathbb{P}}
\newcommand{\finite}{[\omega]^{<\omega}}

\newtheorem{theorem}{Theorem}
\newtheorem{lemma}[theorem]{Lemma}
\newtheorem{proposition}[theorem]{Proposition}

\theoremstyle{definition}
\newtheorem*{definition}{Definition}

\author{Marcin Sabok}\thanks{This research was supported by
  the Mittag-Leffler Institute (Djursholm, Sweden) and by
  the ESF program ``New Frontiers of Infinity: Mathematical,
  Philosophical and Computational Prospects'' through grants
  number 2581 and 2535, and by FWF (Austrian Science Fund)
  grant number P 20835-N13}

\address{Instytut Matematyczny Uniwersytetu Wroc\l awskiego,
  pl. Grunwaldzki $2\slash 4$, $50$-$384$ Wroc\l aw, Poland
}
\address{Kurt G\"odel Research Center for Mathematical
  Logic, W\"ahringer Stra\ss e $25$, A-$1090$ Wien, Austria}
\email{sabok@math.uni.wroc.pl}

\title{Complexity of Ramsey null sets}

\begin{document}

\begin{abstract}
  We show that the set of codes for Ramsey positive analytic
  sets is $\sh$-complete. This is a one projective-step
  higher analogue of the Hurewicz theorem saying that the
  set of codes for uncountable analytic sets is
  $\ana$-complete. This shows a close resemblance between
  the Sacks forcing and the Mathias forcing. In particular,
  we get that the $\sigma$-ideal of Ramsey null sets is not
  ZFC-correct.  This solves a problem posed by Ikegami,
  Pawlikowski and Zapletal.
\end{abstract}



\maketitle

\section{Introduction}

Ramsey measurability was introduced by Galvin and Prikry
\cite{prikry:theorem} to prove a Ramsey theorem for Borel
colorings of the plane. Shortly after, their result was
generalized by Silver \cite{silver:ramsey} to those
colorings of the plane which are in the $\sigma$-algebra
generated by analytic sets. Ellentuck
\cite{ellentuck:ramsey} has later pointed out that Ramsey
measurable sets are precisely the sets with the Baire
property in a certain topology on $\ellentuck$, called today
the \textit{Ellentuck topology}.  The basic open sets in the
Ellentuck topology are of the form
$[\sigma,s]=\{x\in\ellentuck:
x\restriction\max(\sigma)=\sigma\ \wedge\ x\setminus
\max(\sigma)\subseteq s\}$ for $\sigma\in\finite$,
$s\in\ellentuck$ such that $\max \sigma<\min s$. Of crucial
importance is the fact that analytic subsets of $\ellentuck$
have the Baire property in the Ellentuck topology. This
leads to the Silver theorem, saying that every analytic set
$A\subseteq\ellentuck$ is \textit{Ramsey measurable}, i.e.
for any basic open set $[\sigma,s]$ as above there is an
infinite set $s'\subseteq s$ such that $[\sigma,s']$ is
either disjoint from $A$, or contained in $A$. If for any
$[\sigma,s]$ there is an infinite $s'\subseteq s$ such that
$[\sigma,s]$ is disjoint from $A$, then we say that $A$ is
\textit{Ramsey null}. A set is \textit{Ramsey positive} if
it is not Ramsey null.  Note that, by the Silver theorem, an
analytic set is Ramsey positive if and only if it contains
some $[\sigma,s]$ as above. It is worth noting here that the
Silver theorem and the notion of Ramsey measurability have
found many applications outside of set theory, e.g. in the
Banach space theory, cf \cite[Section
19.E]{kechris:classical}. Similar notion appeared also in
the early years of forcing as the \textit{Mathias forcing},
which is the forcing with basic open sets in the Ellentuck
topology, ordered by inclusion.  In an equivalent form, it
can be viewed as the quotient Boolean algebra of Borel
subsets of $\ellentuck$ modulo the $\sigma$-ideal of Ramsey
null sets.

Given a (definable) family $\Phi$ of analytic sets we say
that $\Phi$ is \textit{ZFC-correct} if there is a finite
fragment ZFC$^*$ of ZFC such that for any $A\in\ana$ and any
model $M$ of ZFC$^*$ containing a code for $A$ we have that
$$M\models A\in\Phi\quad\mbox{ if and only if
}\quad V\models A\in\Phi.$$ In fact, ZFC-correctness of
$\Phi$ is equivalent to the fact that the set of codes for
analytic sets in $\Phi$ is provably $\dualsh$.

In \cite{zapletal:idealized} Zapletal developed a general
theory of iteration for idealized forcing. One of the
necessary conditions for a $\sigma$-ideal to be
\textit{iterable} (see \cite[Definition
5.1.3]{zapletal:idealized}) is its ZFC-correctness. This
seems to be very natural assumption since most of the
examples share this property. In fact, many of them,
including the $\sigma$-ideals associated to the Cohen, Sacks
or Miller forcing are \textit{$\coana$ on $\ana$} (see
\cite[Definition 3.8.1]{zapletal:idealized} or
\cite[Definition 25.9]{kechris:classical}), which is even
stronger than ZFC-correctness. Among the few examples which
are known to be ZFC-correct but not $\coana$ on $\ana$ is
the $\sigma$-ideal associated to the Laver forcing.

In \cite{ikegami:absoluteness} Ikegami presented a general
framework of generic absoluteness results for
\textit{strongly arboreal} \cite[Definition
2.4]{ikegami:absoluteness} forcing notions $\P$. Again,
however, an important assumption (cf. \cite[Theorem
4.3]{ikegami:absoluteness}, \cite[Theorem
4.4]{ikegami:absoluteness}) is that the set of Borel codes
for sets in $I^*_\P$ (see \cite[Definition
2.11]{ikegami:absoluteness}) is $\sh$ (for a discussion see
also \cite[Paragraph 7.2]{ikegami:absoluteness}). In the
case $\P$ is the Mathias forcing, $I^*_\P$ is the family of
Ramsey null sets.

Mathias forcing is a natural example of a forcing notion,
for which it was not clear whether the results of
\cite{ikegami:absoluteness} and \cite{zapletal:idealized}
can be applied. This motivated Ikegami, Pawlikowski and
Zapletal to ask whether the $\sigma$-ideal of Ramsey null
sets is ZFC-correct. In this paper we answer this question
negatively.  In fact, we prove the following stonger result,
which seems to be interesting in its own right.

\begin{theorem}\label{main}
  The set of codes for Ramsey positive analytic sets is
  $\sh$-complete.
\end{theorem}

By now, only a few examples of $\sh$-complete sets have been
known. In fact the only source of such sets is
\cite{becker:harmonic}. On the other hand, one level below
in the projective hierarchy, there are lots of natural
examples of $\ana$-complete sets (cf \cite[Section
27]{kechris:classical}). Theorem \ref{main} should be
compared to the Hurewicz theorem \cite[Theorem
27.5]{kechris:classical} saying that the set of codes for
uncountable analytic (or even closed) sets is
$\ana$-complete. Together, these two results show that on
two consequtive levels of the projective hierarchy we
observe a very similar phenomenon. This reveals an analogy
between the Sacks forcing and the Mathias forcing.

\section{Notation}

For a tree $T\subseteq\btree$ we write $\lim T$ for $\{x\in
\baire:\ \forall n\in\omega\ x\restriction n\in T\}$. If
$\tau\in\btree$, then we denote by $[\tau]$ the set
$\{x\in\baire: \tau\subseteq x\}$. Similarly, for
$\tau\in\finite$ we write $[\tau]$ for $\{x\in\ellentuck:
x\restriction\max(\tau)=\tau\}$. For each $n<\omega$ and
$i\in2$ we write $[(n,i)]$ for $\{x\in\cantor: x(n)=i\}$.
For a tree $T\subseteq\btree$ we write $P(T)$ (respectively
$R(T)$) for the set of all perfect (resp. pruned) subtrees
of $T$.  $P(T)$ and $R(T)$ are endowed with Polish
topologies induced via the natural embeddings into
$\cantor$. In particular $P(\ctree)$ stands for the space of
all perfect binary trees.

If $D\subseteq \baire\times\baire$ and $F\subseteq \baire$
are closed, then we write $f:F\continuous D$ to denote that
$f$ is a continuous function from $F$ to $Y$ whose graph is
contained in $D$. Recall \cite[Proposition
2.5]{kechris:classical} that if $T$ and $S$ are trees such
that $F=\lim T$ and $D=\lim S$, then we can code $f$ by a
monotone map from $T$ to $S$, and any monotone map from $T$
to $S$ gives rise to a continuous function defined on a
comeager subset of $F$.

By the \textit{standard topology on $\ellentuck$} we mean
the one induced from the Baire space $\baire$ via the
standard embedding of $\ellentuck$ into $\baire$. Unless
stated otherwise, $\ellentuck$ is always consider as a
topological space with the standard topology. In special
cases we will indicate when we refer to the Ellentuck
topology on $\ellentuck$.

For a sequence of Polish spaces $\langle X_i:i\in I\rangle$
($I$ countable) we write $\bigsqcup_{i\in I}X_i$ for the
disjoint union of the spaces $X_i$ with the natural Polish
topology.

For a Polish space $X$ we write $K(X)$ for the space of
compact subsets of $X$ with the Vietoris topology (cf.
\cite[Section 4.F]{kechris:classical}) and $F(X)$ for the
Polish space of all closed subsets of $X$ (cf.
\cite[Theorem 12.3]{kechris:classical}). Note that if $X$ is
the Baire space $\baire$ (or $\ellentuck$), then the natural
coding of closed sets by pruned trees gives a homeomorphism
of $F(\baire)$ and $R(\btree)$.

All Polish spaces considered in this paper are assumed to be
endowed with a fixed topology subbase. For the Cantor space
$\cantor$ we fix the subbase consisting of the sets
$[(n,0)]$ and $[(n,1)]$ for $n<\omega$. For zero-dimensional
Polish spaces we assume that the fixed subbase is the one
inherited from $\cantor$ via a fixed embedding into
$\cantor$. In particular, the space of all pruned subtrees
of $\btree$ inherits its subbase from $\cantor$ and this
subbase consists of the sets $\{T\in R(\btree): \sigma\in
T\}$ and $\{T\in R(\btree): \sigma\not\in T\}$. Similarly,
the subbase for $F(\ellentuck)$ consists of the sets $\{D\in
F(\ellentuck):\ D\cap[\sigma]\not=\emptyset\}$ and $\{D\in
F(\ellentuck):\ D\cap[\sigma]=\emptyset\}$ for
$\sigma\in\btree$.

By a \textit{pointclass} we mean one of the classes
$\mathbf{\Sigma}^0_\alpha,\mathbf{\Pi}^0_\alpha$ for
$\alpha<\omega_1$ or $\mathbf{\Sigma}^1_n,\mathbf{\Pi}^1_n$
for $n<\omega$. If $\B$ is a Boolean combination of
pointclasses, $X$ and $Y$ are Polish spaces, $\mathcal{U}$
is the fixed subbase for $Y$, and $f:X\rightarrow Y$ is a
function, then we say that $f$ is
\textit{$\B$-submeasurable} if $f^{-1}(U)\in\B$ for each
$U\in\mathcal{U}$. If $\A$ is a pointclass, $A\subseteq X$
is in $\A$ and $f:A\rightarrow Y$ is a function, then we say
that $f$ is \textit{$\A$-measurable} if for each open set
$V\subseteq Y$ there is $B\in\A$ such that $f^{-1}(V)=A\cap
B$. If $Y$ is zero-dimensional and $A\subseteq Y$ is in
$\A$, then we say that $A$ is $(\A,\B)$\textit{-complete} if
for any zero-dimensional Polish space $Z$ and $A'\subseteq
Z$ in $\A$ there is a $\B$-submeasurable function
$f:Z\rightarrow Y$ such that $f^{-1}(A)=A'$. Note that the
notion of $(\A,\mathbf{\Sigma}^0_1)$-completeness coincides
with the usual notion of $\A$-completeness.

Given a pointclass $\A$ and a Polish space $X$ we code the
$\A$-subsets of $X$ using a fixed \textit{good} (cf.
\cite[Section 3.H.1]{moschovakis}) universal $\A$-set
$A\subseteq \cantor\times X$. We refer to $\{x\in \cantor:
A_x\mbox{ is Ramsey null}\}$ as to the set of \textit{codes
  for Ramsey null $\A$ sets}. Note that the complexity of
this set does not depend on the universal set $A$ as long as
$A$ is good.  Recall also that the standard universal sets
for pointclasses are good.

\section{Correctness}

In this section we show that the $\sigma$-ideal of Ramsey
null sets is not ZFC-correct. Recall that the standard
universal $\gdelta$ set $G\subseteq\cantor\times\ellentuck$
is constructed in such a way that if $x\in\cantor$ codes a
sequence of closed subets $\langle D_n:n<\omega\rangle$ of
$\ellentuck$, then
$$G_x=\ellentuck\setminus\bigcup_{n<\omega}D_n.$$ We can
realize this using $\prod_{n<\omega}F(\ellentuck)$ as the
set of codes. The space $\prod_{n<\omega} F(\ellentuck)$ is
embedded (as a $\gdelta$ set) into $\cantor$ using the
pruned trees. We will show that the set of codes for Ramsey
positive $\gdelta$ sets is $(\sh,\ana\cup\coana)$-complete.

Notice that this result is optimal, i.e. the set of codes
for Ramsey positive closed sets (and hence also $\fsigma$
sets) is $\ana$. This follows from the fact that a closed
set $C\subseteq\ellentuck$ is Ramsey positive if and only if
there is a basic open set $[\sigma,s]$ in the Ellentuck
topology such that
$$[\sigma,s]\subseteq C.$$ The latter condition is arithmetical, 
since both sets $[\sigma,s]$ and $C$ are closed in the
standard topology on $\ellentuck$.

Note also that if $B$ is a $\gdelta$ set (or even Borel),
then the condition
$$\exists [\sigma,s]\quad [\sigma,s]\subseteq B$$ is $\sh$
and hence the set of codes for Ramsey positive $\gdelta$
sets is $\sh$.

Since any $\ana\cup\coana$-submeasurable function is
$\dualsh$-measurable, we immediately get that the set
$\{x\in\cantor: G_x$ is Ramsey positive$\}$ is not
$\dualsh$. This implies that the $\sigma$-ideal of Ramsey
null sets is not ZFC-correct, for otherwise we could express
the fact that $G_x$ is Ramsey null as
\begin{eqnarray*}
  \exists M\mbox{ c.t.m. of ZFC}^*\ \ 
  x\in M\ \wedge\ M\models G_x\mbox{ is Ramsey null}
\end{eqnarray*} 
or as 
\begin{eqnarray*}
  \forall
  M\mbox{ c.t.m. of ZFC}^*\ \ x\in
  M\ \Rightarrow\ M\models G_x\mbox{ is Ramsey null},
\end{eqnarray*}
where ZFC$^*$ is a fragment of ZFC recognizing the
correctness of the $\sigma$-ideal of Ramsey null sets.

\begin{theorem}\label{subcomplete}
  The set of codes for Ramsey positive $\gdelta$ subsets of
  $[\omega]^\omega$ is $(\sh,\ana\cup\coana)$-complete.
\end{theorem}

\begin{proof}
  Consider the following set $$Z=\{C\in K(\cantor):\ \
  \exists a\in[\omega]^\omega\ \forall x\in C\ \ \lim_{n\in
    a} x(n)=0\}$$ and recall that $Z$ is $\sh$-complete, by
  a result of Becker, Kahane and Louveau \cite[Theorem
  3.1]{becker:harmonic}. We will find a
  $\ana\cup\coana$-submeasurable reduction from $Z$ to
  $\{x\in\cantor: G_x$ is Ramsey positive$\}$.

  For $C\in K(\cantor)$ and $\tau\in\finite$ we define
  $F_\tau(C)\subseteq\ellentuck$ as follows. Put
  $$F_\tau(C)=\{a\in\ellentuck:\ \big[\neg(\exists x\in C\ \forall n\in
  a\setminus\max(\tau)\quad x(n)=1)\big]\ \vee\
  a\not\in[\tau]\}.$$

\begin{lemma}
  For each $C\in K(\cantor)$ and $\tau\in\ctree$ the set
  $F_\tau(C)$ is open in the standard topology on
  $\ellentuck$.
\end{lemma}

\begin{proof}
  Write $$\bar C=\{(a,x)\in\ellentuck\times\cantor:\ x\in C\
  \wedge\ x\restriction (a\setminus\max(\tau))=1\}$$ and let
  $\pi$ be the projection to $\ellentuck$ from
  $\ellentuck\times\cantor$.  Since $\bar C$ is closed in
  $\ellentuck\times\cantor$, the set $\pi''(\bar C)$ is
  closed in $\ellentuck$. Now we have
  $$\ellentuck\setminus F_\tau(C)=[\tau]\cap\pi''(\bar C).$$
\end{proof}

\begin{lemma}\label{borel}
  For each $\tau\in\ctree$ the function $$K(\cantor)\ni
  C\mapsto \ellentuck\setminus F_\tau(C)\in F(\ellentuck)$$
  is $\ana\cup\coana$-submeasurable.
\end{lemma}

\begin{proof}
  Recall that the subbase for the space $F(\ellentuck)$
  consists of the sets $$\{D\in F(\ellentuck):\
  D\cap[\sigma]\not=\emptyset\},\quad \{D\in F(\ellentuck):\
  D\cap[\sigma]=\emptyset\}$$ for $\sigma\in\btree$. It is
  enough to prove that for each $\sigma\in\btree$ the
  preimage $A_\sigma$ of the set $\{D\in F(\ellentuck):\
  D\cap[\sigma]\not=\emptyset\}$ is $\ana$ in $K(\cantor)$.
  Moreover, it is enough to show this for
  $\sigma\supseteq\tau$. Indeed, $\ellentuck\setminus
  F_\tau(C)$ is always contained in $[\tau]$, so for
  $\sigma\not\supseteq\tau$ we have $A_\sigma=A_\tau$ if
  $\sigma\subseteq\tau$ and $A_\sigma=\emptyset$ otherwise.
  But if $\sigma\supseteq\tau$, then $A_\sigma$ is equal to
  $$\{C\in K(\cantor): \pi''(\bar
  C)\cap[\sigma]\not=\emptyset\},$$ which is the same as
  $$\{C\in K(\cantor):\ \exists x\in C\ \exists
  a\in[\sigma]\quad x\restriction
  (a\setminus\max(\tau))=1\}.$$ The latter set is easily
  seen to be $\ana$.
\end{proof}

Now we define
$F:K(\cantor)\rightarrow\prod_{n<\omega}F(\ellentuck)$ so
that $F(C)=\langle \ellentuck\setminus
F_\tau(C):\tau\in\finite\rangle$ (we use some fixed
recursive bijection between $\omega$ and $\finite$). In
other words, $F(C)$ is the code for the $\gdelta$ set
$$G_{F(C)}=\bigcap_{\tau\in\ctree} \ellentuck\setminus F_\tau(C).$$ Note that,
by Lemma \ref{borel}, the function $F$ is
$\ana\cup\coana$-submeasurable. We will be done once we
prove the following lemma.

\begin{lemma}\label{cplte}
  For $C\in K(\cantor)$ we have $$C\not\in Z\ \ \mbox{if and
    only if}\ \ G_{F(C)}\mbox{ is Ramsey null}.$$
\end{lemma}

\begin{proof}
  ($\Leftarrow$) Suppose $F(C)$ is a code for a Ramsey null
  set. We must show that $C\not\in Z$. Take any
  $a\in\ellentuck$. We shall find $x\in C$ such that
  $$\lim_{n\in a}x(n)\not=0.$$ Since $G_{F(C)}$ is Ramsey null, there is
  $b\subseteq a$, $b\in\ellentuck$ such that
  $$[b]^\omega\cap G_{F(C)}=\emptyset.$$ In particular,
  there is $\tau\in\finite$ such that $b\not\in F_\tau(C)$.
  This means that $$b\in[\tau]\quad\wedge\quad\exists
  x\in C\ \forall n\in b\setminus \max(\tau)\quad x(n)=1.$$
  Hence $x$ is constant $1$ on $b\setminus\max(\tau)$, so
  $\lim_{n\in a}x(n)\not=0$, as desired.

  ($\Rightarrow$) Suppose now that $C\not\in Z$. We must
  show that $F(C)$ is a code for a Ramsey null set. Take any
  $\tau\in\finite$ and $a\in\ellentuck$ such that
  $\max(\tau)<\min(a)$. We shall find $b\in[a]^\omega$ such that
  $$[\tau,b]\cap G_{F(C)}=\emptyset.$$ It is enough to find
  $b\in[a]^\omega$ such that $[\tau,b]\cap
  F_\tau(C)=\emptyset$. Since it is not the case that
  $$\forall x\in C\ \lim_{n\in a} x(n)=0,$$ there is $x_0\in
  C$ and $b\in[a]^\omega$ such that $x_0\restriction b=1$.
  We shall show that $$[\tau,b]\cap F_\tau(C)=\emptyset.$$
  Suppose not. Take any $y\in[\tau,b]\cap F_\tau(C)$. Then
  $y\in[\tau]$, $y\setminus\max(\tau)\subseteq b$ and $y\in
  F_\tau(C)$. So, by the definition of $F_\tau$, we have
  $$\neg(\exists x\in C\ \forall y\setminus\max(\tau)\quad
  x(n)=1).$$ But we saw that $x_0\in C$ and $x_0\restriction
  b=1$, so we have $$x_0\restriction (y\setminus\max(\tau))=1.$$
  This gives a contradiction and shows that $[\tau,b]\cap
  F_\tau(C)=\emptyset$, as required.
\end{proof}
\noindent This ends the proof of the theorem.
\end{proof}

\section{Completeness}

In this section we show the following.

\begin{theorem}\label{completeness}
  Any $(\sh,\ana\cup\coana)$-complete subset of a Polish
  zero-dimensional space is $\sh$-complete.
\end{theorem}

Together with Theorem \ref{subcomplete}, this will prove
Theorem \ref{main}. The proof of Theorem \ref{completeness}
will be based on some ideas of Harrington and Kechris from
\cite{harrington:kechris} and of Kechris from
\cite{kechris:completeness}. 

We will need the following lemma.

\begin{lemma}[Sacks uniformization]\label{sacks}
  Let $Y$ be a Polish space. If $B\subseteq \cantor\times Y$
  is Borel and its projection on $\cantor$ is uncoutable,
  then there is a perfect tree $S\subseteq\ctree$ and a
  continuous function $f:\lim S\rightarrow Y$ such that
  $f\subseteq B$.
\end{lemma}

Zapletal proved \cite[Proposition 2.3.4]{zapletal:idealized}
a general version of the $P_I$-uniformization for any
$\sigma$-ideal $I$ for which the forcing $P_I$ is proper.
The above lemma follows directly from \cite[Proposition
2.3.4]{zapletal:idealized} and the fact that the Sacks
forcing has \textit{continuous reading of names} (see
\cite[Definition 3.1.1]{zapletal:idealized}).

Since the Mathias forcing also has continuous reading of
names, the same uniformization result is true for the
Mathias forcing. In particular, this implies that the set of
codes for $\ana$ Ramsey positive sets is a $\sh$ set.
Indeed, if $A\subseteq\ellentuck$ is $\ana$ and
$D\subseteq\ellentuck\times\baire$ is a closed set
projecting to $A$, then the fact that $A$ is Ramsey positive
can be written as $$\exists[\tau,b]\ \exists
f:[\tau,b]\continuous D\quad f\mbox{ is total}.$$ Saying
that $f$ is total is a $\coana$ statement, which makes the
above $\sh$.

\begin{definition}
  Let $\A$ be a pointclass and $\B$ be a Boolean
  combinations of pointclasses. An
  \textit{$(\A,\B)$-expansion} of a Polish space $Y$ is an
  $\A$-subset $E(Y)$ of a Polish space $Y'$ together with an
  $\A$-measurable map $r:E(Y)\rightarrow Y$ satisfying the
  following. For every zero-dimensional Polish space $X$ and
  $\B$-submeasurable map $f:Y'\rightarrow X$ there is a
  closed (in $Y'$) set $F\subseteq E(Y)$ and a continous map
  $g:Y\rightarrow X$ such that $r''(F)=Y$ and the following
  diagram commutes.
  \begin{displaymath}
    \xymatrix{
      F \ar@{>}[rd]^-{f\restriction F}  \ar@{>}[d]_-{r\restriction F}&   \\
      Y \ar@{>}[r]_-{g} & X
    }
  \end{displaymath}
  Note that in this definition we may demand that
  $X=\cantor$.
\end{definition}

The above notion is relevant in view of the following.

\begin{proposition}\label{use}
  Let $X$ and $Y$ be zero-dimensional Polish spaces,
  $A\subseteq X$ be $(\A,\B)$-complete and $C\subseteq Y$ be
  $\A$-complete. If $Y$ has an $(\A,\B)$-expansion, then $A$
  is $\A$-complete.
\end{proposition}
\begin{proof}
  Let $Y'$, $E(Y)$ and $r:E(Y)\rightarrow Y$ be an
  $(\A,\B)$-expansion of $Y$. Put $C'=r^{-1}(C)$ and note
  that $C'\subseteq Y'$ is also in $\A$. Let
  $f:Y'\rightarrow X$ be $\B$-submeasurable such that
  $f^{-1}(A)=C'$. Take $F\subseteq E(Y)$ and $g:X\rightarrow
  Y$ as in the definition of expansion. Note that
  $g^{-1}(A)=C$.
\end{proof}

In view of Proposition \ref{use} and the fact that there
exists a $\sh$-complete subset of the Cantor space
\cite[Theorem 3.1]{becker:harmonic}, Theorem
\ref{completeness} will follow once we prove the following.

\begin{theorem}\label{expansion}
  There exists a $(\sh,\ana\cup\coana)$-expansion of the
  Cantor space.
\end{theorem}

We will need the following technical result (cf.
\cite[Sublemma 1.4.2]{harrington:kechris}).

\begin{proposition}\label{proposition}
  There exists a $\sh$ set $R\subseteq\cantor$ and a
  $\sh$-measurable function $T:R\rightarrow P(\ctree)$ such
  that for each partition of $\cantor\times\cantor$ into
  $A\in\ana$ and $C\in\coana$ there exists $x\in R$ such
  that
  $$\lim T(x)\subseteq A_x\quad\mbox{or}\quad \lim T(x)\subseteq C_x.$$
\end{proposition}
\begin{proof}
  We begin with a lemma.

  \begin{lemma}\label{lemma}
    Given $x\in\cantor$, for any partition of
    $\omega\times\cantor$ into $A\in\lana(x)$ and
    $C\in\lcoana(x)$ there is a $\ldsh(x)$-recursive
    function $T:\omega\rightarrow P(\ctree)$ such that for
    each $n\in\omega$ we have
    $$\lim T(n)\subseteq A_n\quad\mbox{or}\quad \lim T(n)\subseteq C_n.$$
  \end{lemma}
  \begin{proof}
    Pick a sufficiently large fragment ZFC$^*$ of ZFC and
    consider the set
    \begin{eqnarray*}
      H=\{c\in\cantor: \exists M\mbox{ a countable
        transitive model of ZFC}^*\\ 
      \mbox{ containing }x\mbox{ and }c\mbox{ is a
        Cohen real over }M\}.
    \end{eqnarray*} 
    Since $H$ is $\lsh(x)$, it contains a $\ldsh(x)$ element
    $c$. For each $n<\omega$ both $A_n$ and $C_n$ have the
    Baire property and are coded in any model containing
    $x$. Hence, if $c\in A_n$, then $A_n$ is nonmeager and
    if $c\in C_n$, then $C_n$ is nonmeager.  Put
    $$S=\{n\in\omega: c\in A_n\},\quad P=\{n\in\omega: c\in
    C_n\}$$ and note that both sets $S$ and $P$ are
    $\ldsh(x)$. We shall define the function $T$ on $S$ and
    $P$ separately.

    For each $n\in P$ the set $C_n$ is nonmeager, so in
    particular contains a perfect set. Consider the set
    $$P'=\{(n,T)\in\omega\times P(\ctree): n\in P\
    \wedge\ \lim T\subseteq C_n\}$$ and note that $P'$ is
    $\lsh(x)$. Pick any $\lsh(x)$ uniformization $T_P$ of
    $P'$ and note that $T_P$ is in $\ldsh(x)$.

    For each $n\in S$ the set $A_n$ is nonmeager. Let
    $D\subseteq\omega\times\cantor\times\baire$ be a
    $\lclosed(x)$ set projecting to $A$. Since for $n\in S$
    the set $A_n$ is uncountable, by Lemma \ref{sacks} there
    exists a perfect tree $T$ together with a continuous map
    $h:\{n\}\times\lim T\continuous D_n$. Note that, by
    compactness of $\lim T$, we can code a total continuous
    function on $\{n\}\times\lim T$ using a monotone map.
    Consider the set
    $$S'=\{(n,T)\in\omega\times P(\ctree): n\in S\ \wedge\
    \exists f:\{n\}\times\lim T\continuous D_n\}$$ and note
    that $S'$ is $\lsh(x)$. Pick any $\lsh(x)$
    uniformization $T_S$ of $S'$ and note that $T_s$ is
    $\ldsh(x)$.

    The function $T=T_P\cup T_S$ is as required.
  \end{proof}

  Now we finish the proof of the proposition. Fix a good
  $\lsh$-universal set
  $U\subseteq\omega\times\cantor\times\omega\times\omega$
  such that for each $A\subseteq\omega\times\omega$ and
  $x\in\cantor$ if $A\in\lsh(x)$, then there is $n<\omega$
  such that $$A=U_{(n,x)}.$$ Let $U^*\subseteq U$ be a
  $\lsh$-uniformization of $U$ treated as a subset of
  $(\omega\times\cantor\times\omega)\times\omega$ and write
  \begin{eqnarray*}
    R'=\{(n,x)\in\omega\times\cantor: \forall
    m<\omega\,\exists k<\omega\ (m,k)\in U^{*}_{(n,x)}\ \mbox{
      and }U^{*}_{(n,x)}\\ \mbox{ codes a characteristic function
      of a perfect tree}\},
  \end{eqnarray*}
  where the coding is done via a fixed recursive bijection
  from $\omega$ to $\ctree$. Note that $R'\in\lsh$.

  For $(n,x)\in R'$ we write $\{n\}(x)$ for the perfect tree
  coded by $U^{*}_{(n,x)}$. Note that
  $$(n,x)\mapsto\{n\}(x)$$ is a partial $\lsh$-recursive
  function from $\omega\times\cantor$ to $P(\ctree)$.  

  Now pick a recursive homeomorphism
  $h:\cantor\rightarrow\omega\times\cantor$ and write
  $h(x)=(n_x,x')$. Put $R=h^{-1}(R')$ and $T(x)=\{n_x\}
  (x')$ for $x\in R$.

  We claim that $R$ and $T$ are as required. To see this,
  pick a partition of $\cantor\times\cantor$ into $A\in\ana$
  and $C\in\coana$. Let $z\in\cantor$ be such that
  $A\in\lana(z)$ and $C\in\lcoana(z)$. Let
  $T:\omega\rightarrow P(\ctree)$ be a $\ldsh(z)$-recursive
  function as in Lemma \ref{lemma}. For each $n\in\omega$ we
  have that $T(n)$ is a total $\lsh(z)$-recursive function
  from $\omega$ to $\omega$ coding a perfect tree.
  Therefore, by the Kleene Recursion Theorem for
  $\lsh(z)$-recursive functions \cite[Theorem
  7A.2]{moschovakis} there is $n\in\omega$ such that
  $$T(n)=U_{(n,z)}=\{n\}(z).$$
  Now $x=h^{-1}(n,z)$ has the desired property.
\end{proof}

Now we are ready to prove Theorem \ref{expansion}.

\begin{proof}[Proof of Theorem \ref{expansion}]
  Pick a $\sh$ set $R\subseteq\cantor$ and a
  $\sh$-measurable function $T:R\rightarrow P(\ctree)$ as in
  Proposition \ref{proposition}. For each $x\in R$ let
  $t(x)\in T(x)$ be the first splitting node of $T(x)$ and
  let $T^0,T^1:R\rightarrow P(\ctree)$ be defined as
  $$T^i(x)=T(x)_{t(x)^\smallfrown i}$$ for $i\in 2$. Note
  that $T^0$ and $T^1$ are also $\sh$-measurable.

  For $x\in R$ let
  $$s_x^0,s_x^1:\cantor\rightarrow \lim T^i(x)$$ be 
  induced by the canonical isomorphism of $\ctree$ and
  $T^i(x)$. It is not difficult to see that for each $i\in2$
  the map $(x,y)\mapsto (x,s^i_x(y))$ is a $\sh$-measurable
  function from $R\times\cantor$ to $R\times\cantor$.

  For each $n\in\omega$ let $R_n\subseteq(\cantor)^{n+1}$ be
  defined as
  $$R_n=\{(x_0,\ldots,x_n)\in(\cantor)^{n+1}: 
  x_0\in R\wedge\dots\wedge x_{n-1}\in R\}.$$ For each
  $\tau\in\ctree$ put $X_\tau=(\cantor)^{|\tau|+2}$ and write
  $R_\tau$ for a copy of $R_{|\tau|}$ inside $X_\tau$.

  Pick a homemomorphism
  $q:\cantor\times\cantor\rightarrow\cantor$. For each
  $n\in\omega$ let
  $$p_{n+1}:\bigcupdot_{\tau\in 2^{n+1}} X_\tau\rightarrow
  \bigcupdot_{\tau\in 2^n} X_\tau$$ be a partial function
  such that $\dom(p_{n+1})=\bigcup_{\tau\in 2^{n+1}} R_\tau$
  and if $\tau\in2^{n+1}$, $\tau=\sigma^\smallfrown i$, then
  $p_{n+1}$ maps $R_\tau$ into $R_\sigma$ as follows:
  \begin{displaymath}\label{pn}
    p_{n+1}(x_0,\ldots,x_{n-1},x_n,x_{n+1})=(x_0,\ldots
    x_{n-1},s^i_{x_n}(q(x_n,x_{n+1})))\tag{$*$}    
  \end{displaymath}
  for $(x_0,\ldots,x_{n+1})\in R_\tau$ (the value is treated as a point in
  $R_\sigma$).  Note that each $p_{n+1}$ is $\sh$-measurable
  and $1$-$1$.

  We get the following sequence of spaces and partial
  $\sh$-measurable maps
  $$\cantor\times\cantor=X_\emptyset\xleftarrow{p_1}
  X_{\langle 0\rangle}\cupdot X_{\langle
    1\rangle}\xleftarrow{p_2}\ldots\xleftarrow{p_n}
  \bigcupdot_{\tau\in 2^n} X_\tau \xleftarrow{p_{n+1}}
  \bigcupdot_{\tau\in 2^{n+1}}
  X_\tau\xleftarrow{p_{n+2}}\ldots$$ and we write $t_n$ for
  $p_1\circ\ldots\circ p_n$ for $n>0$ and $t_0$ for the
  identity function on $\cantor\times\cantor$

  Now, let $E(\cantor)\subseteq\cantor\times\cantor$ be
  defined as
  $$E(\cantor)=\bigcap_{n<\omega}\rng(t_n).$$ Notice that
  $E(\cantor)\in\sh$. The map
  $r:E(\cantor)\rightarrow\cantor$ is defined as follows.
  For $n\in\omega$ and $\tau\in2^n$ we put
  \begin{displaymath}
    r(x)\restriction n=\tau\quad\mbox{iff}\quad
  (t_n)^{-1}(x)\in X_\tau.
  \end{displaymath}
  Note that $r$ is $\sh$-measurable.

  We need to check that $E(\cantor)$ and $r$ satisfy the
  properties of expansion. Let $f:\cantor\rightarrow Y$ be
  $\ana\cup\coana$-submeasurable, where $Y$ is a
  zero-dimensional Polish space. Since $Y$ is embedded into
  $\cantor$ and inherits its subbase from $\cantor$ via this
  embedding, we can assume that $Y=\cantor$ and the subbase
  consists of the sets $[(n,i)]$ for $n\in\omega, i\in2$.

  We shall define two trees $\langle
  x_\tau:\tau\in\ctree\rangle$ and $\langle
  u_\tau:\tau\in\ctree\rangle$ such that for each
  $\tau\in\ctree$ and $i\in2$ we have
  \begin{itemize}
  \item $x_\tau\in(\cantor)^{|\tau|+1}$ and $u_\tau\in
    2^{|\tau|+1}$
  \item $u_\tau\subseteq u_{\tau^\smallfrown i}$ and
    $x_\tau\subseteq x_{\tau^\smallfrown i}$,
  \end{itemize}
  and 
  \begin{displaymath}\label{property}
    (f\circ t_n)''(X_\tau|{x_\tau})\subseteq [u_\tau]\tag{$**$}
  \end{displaymath}
  where $X_\tau|{x_\tau}=\{y\in X_\tau:
  y\restriction(|\tau|+1)=x_\tau\ \wedge\ y_{n+1}\in
  T(y_n)\}$.

  Suppose this has been done. Note that then for each
  $n\in\omega$ and $\tau\in2^n$ the sets
  $F_\tau={t_n}''(X_\tau|{x_\tau})$ are closed since, by
  (\ref{pn}), $t_n$ is a continuous function of the last
  variable when the remaining ones are fixed. The sets
  $F_\tau$ form a Luzin scheme of closed sets. Put
  $$F=\bigcap_{n<\omega}\bigcup_{\tau\in 2^n}F_\tau.$$
  We define $g:\cantor\rightarrow\cantor$ so that
  $$g(y)\in\bigcap_{n<\omega}[u_{y\restriction n}].$$ Note
  that $g$ is continuous. From (\ref{property}) we get that
  $g\circ (r\restriction F)= f\restriction F$.

  Now we build the trees $\langle
  x_\tau:\tau\in\ctree\rangle$ and $\langle
  u_\tau:\tau\in\ctree\rangle$. We construct them by
  induction as follows. The two sets
  $$f^{-1}([(0,0)])\quad\mbox{and}\quad f^{-1}([(0,1)])$$
  form a partition $\cantor\times\cantor$ into two sets, one
  of which is $\ana$ and the other $\coana$, by the
  assumption that $f$ is $\ana\cup\coana$-submeasurable. By
  Proposition \ref{proposition} there is $x\in R$ and
  $i\in2$ such that $T(x)\subseteq f^{-1}[[(0,i)]]$. Put
  $x_\emptyset=x$, $u_\emptyset=\langle i\rangle$ and note
  that (\ref{property}) is satisfied.

  Suppose that $n>0$ and $x_\sigma$ and $u_\sigma$ are
  constructed for all $\sigma\in 2^{n-1}$. Fix $\tau\in 2^n$
  and let $\tau=\sigma^\smallfrown i$ for some
  $\sigma\in2^{n-1}$ and $i\in2$. We must find $x_\tau\in
  (\cantor)^{n+1}$ and $u_\tau\in2^{n+1}$.

  Note that the set $\{y\in X_\tau: y\restriction n
  =x_{\sigma}\}$ is homeomorphic to $\cantor\times\cantor$.
  Let $w:\cantor\times\cantor\rightarrow \{y\in X_\tau:
  y\restriction n =x_{\sigma}\}$ denote the canonical
  homeomorphism $y\mapsto {x_\sigma}^\smallfrown y$.
  Consider the partition of $\cantor\times\cantor$ into
  $$(f\circ i_{n+1}\circ w)^{-1}([(n-1,0)])\quad\mbox{and}\quad(f\circ
  i_{n+1}\circ w)^{-1}([(n-1,1)]).$$ One of them is $\ana$
  and the other $\coana$, so by Proposition
  \ref{proposition}, there exists $x\in R$ and $i\in2$ such
  that
  $$T(x)\subseteq (f\circ i_{n+1}\circ
  w)^{-1}([(n-1,i)]).$$ Put $x_\tau={x_\sigma}^\smallfrown
  x$ and $u_\tau={u_\sigma}^\smallfrown i$. To see that
  (\ref{property}) holds note that
  ${p_{n+1}}''(X_\tau|x_\tau)\subseteq X_\sigma|x_\sigma$ by
  the definition (\ref{pn}). Therefore, by the inductive
  assumption we have that $(f\circ
  t_{n+1})''(X_\tau|x_\tau)\subseteq
  [u_\sigma]\cap[(n-1,i)]=[u_\tau]$. This ends the
  construction and the whole proof.

\end{proof}

\section{Acknowledgements}

Part of this work has been done during my stay at the
Institut Mittag-Leffler in the autumn 2009. I would like to
thank Joan Bagaria, Daisuke Ikegami, Stevo Todor\v cevi\'c
and Jind\v rich Zapletal for stimulating discussions and
many useful comments.

\bibliographystyle{plain}
\bibliography{odkazy}

\end{document}